\newif\ifTwoColumn		
\newif\ifTechReport

\TwoColumntrue   

\ifTwoColumn
	\documentclass[review,3p,number]{elsarticle_arXiv}	
\else
		\documentclass[review,3p,number]{elsarticle_arXiv}	
\fi
	
\usepackage{graphicx,color,array,subfigure,url}
\usepackage{amsmath,amsthm,amssymb,amsfonts,booktabs,multirow,bm,balance}
\usepackage{caption}

\theoremstyle{plain}
\newtheorem{theorem}{Theorem}
\newtheorem{lemma}{Lemma}
\newtheorem{proposition}{Proposition}

\newtheorem*{claim*}{Claim}

\theoremstyle{definition}
\newtheorem{definition}{Definition}
\newtheorem{standing}{Standing Assumption}

\theoremstyle{remark}

\newcommand{\R}{\mathbb{R}}

\newcommand{\bbP}{\mathbb{P}}

\begin{document}
\begin{frontmatter}

\ifTwoColumn
	\title{On the Sample Size of Random Convex Programs with \\ Structured Dependence on the Uncertainty (Extended Version)\footnote{This manuscript is an extended author's version of a paper that was accepted for publication in the journal Automatica, which is subject to Elsevier copyright. Changes resulting from the publication process, such as peer review, editing, corrections, structural formatting, and other quality control mechanisms may not be reflected in this document. A definitive
version is published in Automatica (volume 60, pages 182 - 188, 2015); doi: \url{dx.doi.org/10.1016/j.automatica.2015.07.013}. }}
\fi

\author[ethz]{Xiaojing Zhang}
\author[ethz]{Sergio Grammatico}
\author[berkeley]{Georg Schildbach}
\author[ox]{Paul Goulart}
\author[ethz]{John Lygeros}

\address[ethz]{Automatic Control Laboratory, ETH Zurich, Switzerland}
\address[berkeley]{Model Predictive Control Lab, UC Berkeley, United States of America}
\address[ox]{Department of Engineering Science, Oxford University, United Kingdom}

\tnotetext[fninfo]{E-mail addresses: \texttt{xiaozhan@control.ee.ethz.ch} (X. Zhang), \texttt{grammatico@control.ee.ethz.ch} (S. Grammatico), \texttt{schildbach@berkeley.edu} (G. Schildbach), \texttt{paul.goulart@eng.ox.ac.uk} (P. Goulart), \texttt{jlygeros@control.ee.ethz.ch} (J. Lygeros).}

\begin{abstract}
The ``scenario approach" provides an intuitive method to address chance constrained problems arising in control design for uncertain systems. It addresses these problems by replacing the chance constraint with a finite number of sampled constraints (scenarios). The sample size critically depends on Helly's dimension, a quantity always upper bounded by the number of decision variables. However, this standard bound can lead to computationally expensive programs whose solutions are conservative in terms of cost and violation probability. We derive improved bounds of Helly's dimension for problems where the chance constraint has certain structural properties. The improved bounds lower the number of scenarios required for these problems, leading both to improved objective value and reduced computational complexity. 
Our results are generally applicable to Randomized Model Predictive Control of chance constrained linear systems with additive uncertainty and affine disturbance feedback. The efficacy of the proposed bound is demonstrated on an inventory management example.
\end{abstract}

\end{frontmatter}


\section{Introduction}\label{sec:introduction}
Many problems in systems analysis and control synthesis can be formulated as optimization problems, including Lyapunov stability, 
robust control, and Model Predictive Control (MPC) problems
\ifTwoColumn
   	\cite{garcia:prett:morari:89, maciejowski, boyd:elghaoui:feron:balakrishnan, tempo:calafiore:dabbene, MayneMPC2014}.
\fi
In reality, most systems are affected by uncertainty and/or disturbances, in which case a control decision should be made that accounts for these uncertainties. In robust optimization, one seeks a solution satisfying all admissible uncertainty realizations (worst-case approach). Unfortunately, robust programs are in general difficult to solve \cite{bental:nemirovski:98}, and computational tractability is often obtained at the cost of introducing conservatism.

Stochastic optimization offers an alternative approach where constraints are interpreted in a probabilistic sense as chance constraints, allowing for constraint violations with limited probability \cite{prekopa}.
Except for special cases, chance constrained problems are also intractable, since they generally are non-convex and require the computation of high-dimensional probability integrals. 
Randomized methods are tools for approximating the solutions to such problems, without being limited to specific probability distributions. 
By replacing the chance constraint with a finite number of randomly sampled constraints, the fundamental question in randomized algorithms is how large to choose the sample size to guarantee constraint satisfaction with high confidence. 
One approach is based on the Vapnik-Chervonenkis (VC) theory of statistical learning \cite{ anthony:biggs}, which has been studied widely for control applications \cite{tempo:calafiore:dabbene, vidyasagar:01}. In general, however, statistical learning theory comes with high sample complexity and consequent conservatism \cite[Section 1.2]{calafiore:campi:05}. 
Recently, a new randomized method known as the ``scenario approach'' has emerged, which is applicable whenever the sampled program is convex \cite{campi:garatti:08, calafiore:10}, and which has been successfully exploited for control design, both theoretical and applied, see e.g.\ \cite{calafiore:campi:06, zhang:schildbach:sturzenegger:morari:13, parisioIFAC2014implementation, ZhangECC2015}. Compared to methods based on the theory of statistical learning, the sample size required by the scenario approach is typically much lower \cite[Section 1.2]{calafiore:campi:05}.

The sample size bounds provided by the scenario approach are based on the notion of \emph{Helly's dimension} \cite[Definition 3.3]{calafiore:10}, which is always upper bounded by the number of decision variables \cite[Lemma 2.2]{calafiore:10}. Since the sample size bound grows linearly in Helly's dimension \cite[Corollary 5.1]{calafiore:10}, finding better bounds not only reduces conservatism of the solution, but also allows problems to be solved faster. Unfortunately, computing  Helly's dimension for a given problem is often challenging. 
To the best of the authors' knowledge, the only attempts to obtain improved bounds are the works in \cite{schildbach:fagiano:morari:13, zhang:grammatico:schildbach:goulart:lygeros:14ecc}. It is shown in \cite{schildbach:fagiano:morari:13} that Helly's dimension can be upper bounded by the so-called \textit{support rank} (s-rank), obtained by exploiting structural properties of the constraints in the \emph{decision space}. Under certain technical assumptions, the authors of \cite{ zhang:grammatico:schildbach:goulart:lygeros:14ecc} upper bound Helly's dimension by the number of active constraints, which can applied for cases where the constraint functions are affine in the uncertainty variables.

In this paper, we propose new methodologies for bounding Helly's dimension that exploit additional structure in the constraint functions. We first establish bounds for generic problems where the constraint functions are separable in the decision and uncertainty variables. We then exploit these structures for cases where the constraint functions depend affinely and quadratically on the uncertainty variables. 
The derived sample size depends on the dimension of the \emph{uncertainty space}, hence complementing \cite{schildbach:fagiano:morari:13} and generalizing those in \cite{zhang:grammatico:schildbach:goulart:lygeros:14ecc}. Furthermore, we also show explicitly that for the considered problems the scenario approach together with our bounds always provides lower sample sizes than the corresponding ones based on the VC theory of statistical learning.

\ifTwoColumn
	The paper is organized as follows: Section \ref{sec:background} establishes the problem setup and presents the technical background. In Section \ref{sec:structuredUncertainty} we present the case when the constraint function exhibits a \emph{structured} dependence on the uncertainty. We compare the new sample sizes obtained with existing methods in Section \ref{sec:sampleSize}. In Section \ref{sec:Examples}, we apply our bounds to Randomized MPC problems, and present numerical results for an inventory management problem. Finally, Section \ref{sec:conclusion} concludes the paper. 
\fi

\section{Problem Description and Technical Background}\label{sec:background}
Let $\delta\in\Delta\subseteq\R^d$ be a random variable defined on a probability space $(\Delta,\mathcal{F},\bbP)$, where $\Delta$ is the sample space, $\mathcal{F}$ a $\sigma$-algebra on $\Delta$, and $\mathbb{P}$ the probability measure defined on $\mathcal{F}$. 
We consider chance constrained problems (CCPs) of the form
\begin{equation}\label{eq:CCP_generic}
\textsc{CCP}(\epsilon): \ \left\{
\begin{array}{l}
\displaystyle \min_{x \in \mathcal{X}}\quad c^\top x \\
\text{s.t. }\quad  \bbP \left[ g(x,\delta) \leq 0  \right] \geq 1 - \epsilon,
\end{array}
\right.
\end{equation}
where $\mathcal{X}\subset\R^n$ is a compact convex set, $x\in\R^{n}$ the decision variable, $g: \R^n \times \Delta \to \R$ the constraint function, $\epsilon\in(0,1)$ the acceptable violation probability, and $c\in\R^n$ the cost vector.  We consider the scenario program (SP) associated with $\textsc{CCP}(\epsilon)$, where the chance constraint in \eqref{eq:CCP_generic} is replaced by $N$ sampled constraints, corresponding to independent identically distributed (i.i.d.) realizations $\delta^{(1)},\ldots,\delta^{(N)} \in \Delta$ of the uncertainty vector $\delta$ \cite{calafiore:campi:06, calafiore:campi:05}:
\begin{equation} \label{eq:SP_generic}
\textsc{SP}[\omega]: \ \left\{
\begin{array}{l}
\displaystyle \min_{ x \in \mathcal{X} }\quad \ c^\top x \\
\text{s.t. } \quad \ g(x,\delta^{(j)}) \leq 0 \quad \forall j\in\{1,\ldots,N\}.
\end{array}
\right.
\end{equation}
We refer to $\omega:=\{\delta^{(1)},\ldots,\delta^{(N)}\}\in \Delta^N$ as a multi-sample. Throughout this paper, we make the following assumption.
\begin{standing}[Regularity]\label{sa:convexity_uniqueGlobalMinimizer}
For almost all $\delta\in\Delta$, the function $x \mapsto g(x,\delta)$ is \emph{convex} and \emph{lower semi-continuous}. For any integer $N$, $\textsc{SP}[\omega]$ in \eqref{eq:SP_generic} is almost surely \emph{feasible}; its optimizer \emph{exists} and is \emph{unique} for almost all realizations of $\omega\in\Delta^N$. For all $x\in\mathcal{X}$, the mapping $\delta \mapsto g(x,\delta)$ is measurable.
\end{standing}

Standing Assumption \ref{sa:convexity_uniqueGlobalMinimizer} is standard in the scenario approach  \cite[Assumption 1]{campi:garatti:08}, \cite[Assumptions 1, 2]{calafiore:10}, \cite[Assumptions 1, 2]{calafiore:campi:06}, and \cite[Appendix B]{grammatico:zhang:margellos:goulart:lygeros:15}. The uniqueness requirement can be relaxed by adopting a suitable (strictly convex or lexicographic) tie-break rule \cite[Section 4.1]{calafiore:campi:05}. We refer to \cite[Appendix B]{grammatico:zhang:margellos:goulart:lygeros:15} and \cite{mohajerin:sutter:lygeros:14} for a technical discussion on measurability issues.

Let us denote the (unique) minimizers of $\textsc{SP}[\omega]$ and $\textsc{SP}[\omega \setminus \{\delta^{(k)}\}]$, for $k \in \{ 1, ..., N\}$, by $x^\star$ and $x^\star_k$, respectively. Our forthcoming results are based on the following two key definitions.
\begin{definition}[{Support Constraint \cite[Definition 4]{calafiore:campi:05}}]
The sample $\delta^{(k)}$ is called a \emph{support sample} if $c^\top x_k^\star < c^\top x^\star$; in this case the corresponding constraint $g(\cdot, \delta^{(k)})$ is called a \textit{support constraint} for $\textsc{SP}[\omega]$.
The \textit{set of support constraints} of $\textsc{SP}[\omega]$ is denoted by $\text{sc}(\textsc{SP}[\omega])$.
\end{definition}
\ifTwoColumn
   	{}
\fi
We denote by $|\text{sc}(\textsc{SP}[\omega])|$ the cardinality of the set of support constraints. 
\begin{definition}[{Helly's dimension \cite[Definition 3.1]{calafiore:10}}]
	\emph{Helly's dimension} of $\textsc{SP}[\omega]$ in \eqref{eq:SP_generic} is the smallest integer $\zeta$ such that 
	\begin{equation*}
		\displaystyle \operatorname{ess}\sup_{\omega\in\Delta^N}\ |\text{sc}(\textsc{SP}[\omega])|\leq \zeta
	\end{equation*}
holds for any finite $N\geq1$.
\end{definition}
\ifTwoColumn
   	Helly's dimension thus is the maximum number of support constraints $|\text{sc}(\textsc{SP}[\omega])|$ for any possible realization of a multi-sample. 
\fi
Intuitively, the \textsc{SP} in \eqref{eq:SP_generic} can be used to approximate the \textsc{CCP} in \eqref{eq:CCP_generic}.
Indeed, the authors of \cite{campi:garatti:08, calafiore:10} show that 
if the sample size $N$ satisfies
\begin{equation}\label{eq:N_choice_zeta}
	\sum_{j=0}^{\zeta-1} {N \choose j} \epsilon^j (1-\epsilon)^{N-j} \leq \beta 
\end{equation}
for some $\beta\in(0,1)$, then, with confidence at least $1-\beta$, the optimal solution of $\textsc{SP}[\omega]$ is feasible for the original $\textsc{CCP}(\epsilon)$ \cite[Theorem 3.3]{calafiore:10}. 
\ifTwoColumn
   	{}
\fi
It was shown in \cite[Lemma 2.2]{calafiore:10} that Helly's dimension $\zeta$ is always upper bounded by $n$. This \emph{standard bound} ($\zeta \leq n$), however, is only tight for fully-supported problems \cite[Theorem 1]{campi:garatti:08}, but remains conservative otherwise. 

The overall goal of this paper is to find tighter upper bounds on Helly's dimension $\zeta$ for non-fully-supported problems, which would allow for smaller $N$ than the one given by the standard bound. Following \cite{calafiore:10} one can show that \eqref{eq:N_choice_zeta} is satisfied if $N$ is chosen such that
\begin{equation}\label{eq:N_expl_cal}
	 N \geq \frac{2}{\epsilon} \left( \zeta-1 + \ln\left(\frac{1}{\beta}\right)  \right).
\end{equation}

Since $\epsilon$ is typically chosen small in many practical applications (e.g.\ $10^{-1} \sim 10^{-4}$)  and $N$ roughly scales as $\mathcal{O}(\zeta/\epsilon)$, finding a good bound on $\zeta$ is key for reducing the required sample size. A small sample size is attractive mainly for two reasons: less conservative solutions in terms of cost, and reduced computational time for solving the scenario program in \eqref{eq:SP_generic}. Also, a small $N$ is beneficial for cases in which the extraction of samples is itself costly. 

\subsection{Bounding Helly's dimension}
Unfortunately, explicitly computing $\zeta$ is in general very difficult. Tighter bounds on Helly's dimension were introduced in \cite{schildbach:fagiano:morari:13}, based on the so-called \textit{support rank} (s-rank). It is defined as the dimension $n$ of the decision space minus the dimension of the maximal linearly unconstrained subspace \cite[Definition 3.6]{schildbach:fagiano:morari:13}, and is therefore never worse than the standard bound.
\ifTwoColumn
   	As shown in \cite[Example 3.5]{schildbach:fagiano:morari:13}, the s-rank can be explicitly computed if $g(x,\delta)$ is affine or quadratic in the decision variable $x$, and hence in some cases can lead to a significant reduction in the sample size $N$.
\fi

There are, however, cases where the s-rank yields no improvement upon the standard bound, although the exact Helly's dimension is much lower. Consider, for instance,
\begin{equation}\label{eq:illExample}
\left\{
\begin{array}{c l}
\displaystyle	\min_{(y,h)\in\R^{n-1}\times\R} \quad & h \\
	\text{s.t.}  \displaystyle  & \bbP \left[ \|A y -b\| +\delta  \leq h \right] \geq 1-\epsilon, 
\end{array}
\right.
\end{equation}
where $\|\cdot\|$ is any norm, $\delta\in\R$ is a continuous random variable, $A\in\R^{k\times (n-1)}$ has full column rank and $b\in\R^k$. The s-rank for the above problem is $n$, because $A$ is full column rank. Hence, the s-rank does not improve upon the standard bound on Helly's dimension, resulting in the same bound $\zeta\leq n$. However, by exploiting structural dependence of the constraint function on the uncertainty, we will show in Section \ref{sec:additiveDependence} that the number of support constraints of any SP associated with the CCP in \eqref{eq:illExample} is equal to $1$ almost surely. 

More generally, this paper addresses the problem of upper bounding Helly's dimension in the case $g(x,\delta)$ does not offer enough structure in $x$ for the s-rank to improve upon the standard bound. Unless stated otherwise, the proofs of all following statements can be found in  \ref{app: proofs}.

\section{Structured Dependence on the Uncertainty} \label{sec:structuredUncertainty}

We consider the case where the constraint function of the CCP in \eqref{eq:CCP_generic} is structured and vector-valued of the form $g:\R^n\times \Delta \to \R^r$, where we interpret the inequalities $g(x,\delta)\leq0$ element-wise. All the definitions and results from the previous section carry over by considering the scalarized function $\max_{i\in\{1,\ldots,r\}}\{g_i(x,\delta)\}$, where $g_i(\cdot,\cdot)$ is the $i$th component of $g(\cdot,\cdot)$. Let us now assume that $g(\cdot,\cdot)$ has a \emph{separable} (``sep" for short) structure of the form
\ifTwoColumn
   	\begin{equation}\label{eq:general_separable_constraints}
		g(x,\delta) = g_\text{sep}(x,\delta) := G(x)\, q(\delta) + H(x) + s(\delta),
	\end{equation}
\fi
where $G: \R^n \to \R^{r\times m}$, $q: \Delta \to \R^m$, $H: \R^n \to \R^r$, and $s: \Delta \to \R^r$.
\ifTwoColumn
   {}
\fi
The following statement contains the fundamental result of this paper.

\ifTwoColumn
\begin{lemma}\label{lem:convex_sep_r_sc}
	Let $g(x,\delta)$ be as in \eqref{eq:general_separable_constraints}. Then  $|\textnormal{sc}(\textsc{SP}[\omega])| \leq r(m+1)$ almost surely.
\end{lemma}
\fi

The proof relies on Radon's Theorem and resembles proofs used in statistical learning theory when determining the VC-dimension of hyperplanes \cite[Theorem 7.4.1]{anthony:biggs}. Intuitively speaking, the result can be explained by introducing auxiliary optimization variables $y_1 := G(x)\in\R^{r \times m}$ and $y_2 := H(x)\in\R^r$, implying that Helly's dimension is upper bounded by $r(m+1)$, the total number of auxiliary optimization variables.

\ifTwoColumn
   {}
\fi

\subsection{Multiplicative dependence on the uncertainty}
Let us now consider the case in which the uncertainty $q(\delta)$ enters \emph{multiplicatively} in the constraints, that is
\ifTwoColumn
   	\begin{equation}\label{eq:SP_mult_sc}
		g(x,\delta) = g_\text{mult}(x,\delta) := G(x)\, q(\delta) + s(\delta).
	\end{equation}
\fi
This is a special case of \eqref{eq:general_separable_constraints} and can be obtained by setting $H(x)=0$. 
\ifTwoColumn
   	{}
\fi
\ifTwoColumn
   	Then the number of support constraints with $g(x,\delta)=g_\text{mult}(x,\delta)$ is bounded as follows.
	\begin{proposition}\label{prop:sc_mults}
		Let $g(x,\delta)$ be as in \eqref{eq:SP_mult_sc}. Then $|\textnormal{sc}(\textsc{SP}[\omega])| \leq rm$ almost surely.
	\end{proposition}
\fi

Note that the results in Proposition \ref{prop:sc_mults} and Lemma \ref{lem:convex_sep_r_sc} also hold if a constant vector  $a\in\R^r$ is appended in the constraint functions, e.g.\ $g(x,\delta)=G(x)\, q(\delta) + s(\delta) + a $.

\subsection{Additive dependence on the uncertainty}\label{sec:additiveDependence}
We consider now the case in which the uncertainty enters \emph{additively} in the constraints, that is
\ifTwoColumn
   	\begin{equation}\label{eq:SP_add_sc}
		g(x,\delta) = g_\text{add}(x,\delta) := H(x) + s(\delta).
	\end{equation}
\fi

\ifTwoColumn
   	{}
\fi
\ifTwoColumn
	\begin{proposition}\label{prop:sc_add}
	Let $g(x,\delta)$ be as in \eqref{eq:SP_add_sc}.	Then $|\textnormal{sc}(\textsc{SP}[\omega])| \leq r$ almost surely.
	\end{proposition}
\fi
\begin{proof}
	Invoke Lemma \ref{lem:convex_sep_r_sc} with $G(x)=0$ in \eqref{eq:general_separable_constraints}. 
\end{proof}

The bound on the number of support constraints for the example in \eqref{eq:illExample} is now readily derived from Proposition \ref{prop:sc_add} with $r=1$ and by realizing that the minimizer in \eqref{eq:illExample} is finite and exists for all multi-sample realizations.

\subsection{Affine dependence on the uncertainty}
We now consider the special case where $g(x,\delta)$ consists of $r$ constraints, where each of them depends affinely on $\delta$. That is, with $q(\delta)=\delta$ and for given  $G: \R^{n} \to \R^{r\times d}$ and $H: \R^n \to \R^r$, we consider 
\ifTwoColumn
   	\begin{equation}\label{eq:SP_affine_sc}
		g(x,\delta) = g_\text{aff}(x,\delta) := G(x)\delta + H(x).
	\end{equation}
\fi
\ifTwoColumn
	\begin{proposition}\label{prop:convex_affine_r_sc}
	Let $g(x,\delta)$ be as in \eqref{eq:SP_affine_sc}. Then $|\textnormal{sc}(\textsc{SP}[\omega])| \leq \zeta_\textnormal{aff} := r(d+1)$ almost surely.
	\end{proposition}
\fi
\begin{proof}
	Invoke Lemma \ref{lem:convex_sep_r_sc} with $q(\delta)=\delta$, $s(\delta)=0$, and $m=d$.
\end{proof}
Proposition \ref{prop:convex_affine_r_sc} combined with \cite[Theorem 3.3]{calafiore:10} immediately leads to the following result.
\ifTwoColumn
   	\begin{theorem}\label{thm:convex_affine_r_sc}
	Let $g(x,\delta)$ be as in \eqref{eq:SP_affine_sc}. For all $\epsilon, \beta\in(0,1)$, if $N$ satisfies \eqref{eq:N_choice_zeta} with $r(d+1)$ in place of $\zeta$, then with confidence no smaller than $1-\beta$, the solution of $\textsc{SP}[\omega]$ in \eqref{eq:SP_generic} is feasible for $\textsc{CCP}(\epsilon)$ in \eqref{eq:CCP_generic}.
	\end{theorem}
\fi

Whenever  $r(d+1)$ is lower than the s-rank, then our result improves on existing sample complexity bounds.

\subsection{Quadratic dependence on the uncertainty}\label{sec:quadraticDependence}
We now extend the previous results to the case of quadratic mapping $\delta \mapsto g(x,\delta)$. Considering $r\geq 1$ constraints, we have functions $g_i:\R^n\times\Delta\to\R$, with $i\in\{1,\ldots,r\}$, defined as $ g_i(x,\delta) := \delta^\top A_i(x) \delta + b_i(x)^\top \delta + c_i(x)$, where $A_i: \R^n\to\R^{d \times d},\ b_i: \R^n\to\R^d$, and $c_i:\R^n\to\R$. Without loss of generality, for all $i\in\{1,\ldots,r\}$ and all $x\in\R^n$, we can assume $A_i(x)=A_i(x)^\top$. Define now 
\ifTwoColumn
   	\begin{equation}\label{eq:SP_quadratic_sc}
		g(x,\delta) = g_\text{quad}(x,\delta) := \max_{i\in\{1,\ldots,r\}} g_i(x,\delta).
	\end{equation}
	\ifTwoColumn
	\else
		We now upper bound the number of support constraints with $g(x,\delta)=g_\text{quad}(x,\delta)$, and derive the corresponding sample complexity.
	\fi
\fi
\ifTwoColumn
	\begin{proposition}\label{prop:convex_quad_r_sc}
	Let $g(x,\delta)$ be as in \eqref{eq:SP_quadratic_sc}. Then  $|\textnormal{sc}(\textsc{SP}[\omega])| \leq \zeta^\textnormal{quad} := rd(d+3)/2+r$ almost surely.
	\end{proposition}
	It now follows immediately from \cite[Theorem 3.3]{calafiore:10}:
	\begin{theorem}\label{thm:convex_quad_r_sc}
		Let $g(x,\delta)$ be as in \eqref{eq:SP_quadratic_sc}. For all $\epsilon\in(0,1)$ and $\beta\in(0,1)$, if $N$ satisfies \eqref{eq:N_choice_zeta} with $rd(d+3)/2+r$ in place of $\zeta$, then with confidence no smaller than $1-\beta$, the solution of $\textsc{SP}[\omega]$ in \eqref{eq:SP_generic} is feasible for $\textsc{CCP}(\epsilon)$ in \eqref{eq:CCP_generic}.
	\end{theorem}
\fi

If $rd(d+3)/2+r$ is smaller than the s-rank, then our sample complexity improves on all known bounds.

\subsection{Box constraints}\label{sec:BoxConstraints}
Let us now study the case in which the constraint function $g(x,\delta)$ is subject to ``box constraints", i.e.\ $g(x,\delta)$ is simultaneously upper and lower bounded. Such structure typically appears in control applications where upper and lower bounds are imposed on states and inputs. It can be shown in this case that the bounds derived in the previous subsections still hold. We  show this for $g_\text{mult}(x,\delta)$.

Assuming that  $x\mapsto g(x,\delta)$ is affine almost surely and given two vectors $\underline{g},\bar{g}\in\R^r$,  consider the constraints
\begin{equation}\label{eq:SP_mult_sc_box} 
	\underline{g} \leq	g_\text{mult}(x,\delta)  \leq \bar{g},
\end{equation}
with $g_\text{mult}(\cdot,\cdot)$ as in \eqref{eq:SP_mult_sc}. Then the following holds.
\begin{proposition}\label{prop:sc_mults_box}
		Consider the constraints in \eqref{eq:SP_mult_sc_box} with $g_\textnormal{mult}(\cdot,\cdot)$ as in \eqref{eq:SP_mult_sc}. Then $|\textnormal{sc}(\textsc{SP}[\omega])| \leq rm$ almost surely.
\end{proposition}
It follows from the proof of Proposition \ref{prop:sc_mults_box} that similar results can be obtained if $g_\text{sep}(x,\delta)$, $g_\text{add}(x,\delta)$, $g_\text{aff}(x,\delta)$ and $g_\text{quad}(x,\delta)$ are box constrained, with the bounds ${r(m+1)}$, $r$, $r(d+1)$ and $rd(d+3)/2+r$, respectively.

\section{Discussion}\label{sec:sampleSize}
\ifTwoColumn
   We first comment on the connection between our bound and the s-rank. Then, we discuss and compare our bound with other bounds.
\fi

\subsection{Connection to s-rank}\label{subsec:sRank}
In general, the sample size bounds based on the s-rank and our derived bounds are not directly comparable, since the former relies on structure in the \emph{decision space}, whereas our bounds exploit structure in the \emph{uncertainty space}.  As shown in \cite[Example 3.5]{schildbach:fagiano:morari:13}, the s-rank can be explicitly computed if $g(x,\delta)$ is affine or quadratic in $x$, whereas our bounds are useful whenever $g(x,\delta)$ is affine or quadratic in $\delta$ (Theorems \ref{thm:convex_affine_r_sc} and \ref{thm:convex_quad_r_sc}).
\ifTwoColumn
   	More specifically, the s-rank can be readily computed for constraints of the form $(x-x_c(\delta))^\top Q (x-x_c(\delta)) - r(\delta) \leq 0$, while it is not trivial to do so for $\delta^\top A(x) \delta + b(x)^\top \delta + c(x) \leq 0$, and vice versa. There are cases, however, where the s-rank and our bounds coincide. This is, for example, the case when the constraint function is bilinear in $x$ and $\delta$.
\fi
In practice, the minimum of the s-rank and our bound should be taken to obtain a tighter bound on Helly's dimension. We will come back to this in Section  \ref{sec:Examples}.

\subsection{Comparison with Statistical Learning Theory}
\ifTwoColumn
   	A different randomized method for approximating $\text{CCP}(\epsilon)$ is based on statistical learning theory \cite{tempo:calafiore:dabbene, anthony:biggs, vidyasagar:01}. It is applicable for constraint functions of the form \eqref{eq:SP_affine_sc} and \eqref{eq:SP_quadratic_sc}, in which case the sample sizes scale as $\mathcal{O}\left(\frac{\xi^\text{vc}}{\epsilon}\ln\frac{1}{\epsilon}\right)$, where $\xi^\text{vc}$ is a bound on the VC-dimension\footnote{Recall that the sample size required by the scenario approach scales as $\mathcal{O}(\zeta/\epsilon)$.}. Specifically, known bounds for the VC-dimension are $\xi^\text{vc}_\text{aff} = 2r\log_2(er)(d+1)$ and $\xi^\text{vc}_\text{quad} = 2r\log_2(er)(d(d+3)/2+1)$ \cite{anthony:biggs}. From Propositions \ref{prop:convex_affine_r_sc} and \ref{prop:convex_quad_r_sc} it follows that $\zeta_\text{aff} < \xi^\text{vc}_\text{aff}$ and $\zeta_\text{quad} < \xi^\text{vc}_\text{quad}$ for any $r$ and $d$. Therefore, we  conclude that  our proposed sample sizes are lower than those based on statistical learning theory for these two cases.
	
	This is not surprising, since statistical learning theory can be applied to any (non-convex) optimization program with finite VC-dimension, whereas the scenario approach is typically restricted to random convex programs. However, we remark here that recent results related to scenario-based optimization have extended the standard (convex) scenario approach to certain classes of random non-convex programs, see e.g.\ \cite{grammatico:zhang:margellos:goulart:lygeros:15, grammatico:zhang:margellos:goulart:lygeros:acc14, calafiore:lyons:fagiano:12, mohajerin:sutter:lygeros:14} for theoretical results and \cite{zhang:grammatico:margellos:goulart:lygeros:14ifac} for an application towards Randomized Nonlinear MPC.
\fi

\subsection{Connection to other results}\label{sec:connectionOtherResults}
Scenario programs subject to constraints of form \eqref{eq:SP_affine_sc} were recently studied in \cite{zhang:grammatico:schildbach:goulart:lygeros:14ecc} in the context of Randomized Model Predictive Control (RMPC). Under the (more restrictive) assumptions of an absolutely continuous probability distribution and that $G(x^\star)\neq0$, it was shown that $|\textnormal{sc}(\textsc{SP}[\omega])| \leq rd$, compared to $r(d+1)$ in Proposition \ref{prop:convex_affine_r_sc}. 
\ifTwoColumn
   	The result presented in this paper, however, applies to a wider range of problems, without any assumption on the underlying probability distribution, at the cost of a slight increase in the upper bound for Helly's dimension.
\fi

An alternative approach to approximate CCP$(\epsilon)$ was proposed in \cite{Margellos2014}, and  studied in \cite{zhang:margellos:goulart:lygeros:13, zhang:georghiou:lygeros:15} in the context of  Stochastic MPC. The method is based on a combination of scenario and robust optimization. Similar to our results here, the sample sizes in \cite{Margellos2014} depend on the dimension of the uncertainty space. However, since the method requires the solution of a robust problem, where the constraints need to be satisfied robustly for all uncertainty realizations inside a predefined set, the solutions can become substantially more conservative than the standard scenario approach, see  e.g.\ \cite[Section V.C]{Margellos2014} and \cite[Section V.B]{zhang:margellos:goulart:lygeros:13}. Nevertheless, recent results suggest that some conservatism can be overcome by using non-linear policies \cite{zhang:georghiou:lygeros:15}. This issue is subject to future research.

\section{Randomized Model Predictive Control of Chance Constrained Systems}\label{sec:Examples}
In this section, we apply our new bounds to Randomized Model Predictive Control (RMPC) as a method for approximating chance constrained Stochastic MPC problems \cite{CalafioreFagTAC2013, SchildbachAutomatica2014, zhang:grammatico:schildbach:goulart:lygeros:14ecc}. 
We consider linear systems of the form 
$${\textstyle x_{k+1} = Ax_k + Bu_k + E\delta_k,}$$
where $k\in\{0,\ldots,T-1\}$ is the time index, $T$ the prediction horizon, $x_k\in\mathbb{R}^{n_x}$ is the state, $u_k\in\mathbb{R}^{n_u}$ is the input, and $\delta_k\in\Delta\subset\R^{n_\delta}$ is an uncertain disturbance. For given matrices $G\in\R^{n_g\times n_u}$, $g\in\R^{n_g}$, $F\in\R^{n_f\times n_x}$ and $f\in\R^{n_f}$, we consider state and input constraints of the form $u_k\in\mathbb{U}:=\{u\in\R^{n_u}:Gu\leq g\}$ and $x_k\in\mathbb{X}:=\{x\in\mathbb{R}^{n_x}:Fx\leq f\}$, where $n_f$ ($n_g$) is the number of state (input) constraints. We require  the state constraints to be satisfied with probability at least $1-\epsilon$. To account for the uncertainties entering the system, we model the inputs using affine decision rules \cite{ben2004adjustable, Goulart2006}:
\begin{equation}\label{eq:input_k}
	u_k(\bm\delta) := h_k + \sum_{j=0}^{k-1} M_{k,j}\delta_j,
\end{equation}
where $h_k\in\R^{n_u}$ and $M_{k,j}\in\R^{n_u\times n_\delta}$ are decision variables, $\bm\delta:=[\delta_0,\ldots,\delta_{T-1}]$, and the (empty) sum is defined as $\sum_{j=0}^{-1}(\cdot):=0$. Note that, to ensure causality, the input at time $k$ depends only on the past disturbances $\delta_0,\ldots,\delta_{k-1}$. We assume for \eqref{eq:input_k} that $\bm\delta$ can be measured directly. As the control objective, any convex cost function $J(\mathbf{u})$ can be chosen, where $\mathbf{u}:=[u_0,\ldots,u_{T-1}]$.
If $x_0$ is the initial state, $\mathbf{h}:=[h_0,\ldots,h_{T-1}]$ and $\mathbf{M}\in\R^{Tn_u\times Tn_\delta}$ is the collection of all $\{M_{k,j}\}_{k,j}$, then the predicted state at stage $k$ is
\begin{equation}\label{eq:state_k}
	x_k(\bm\delta) = A^k x_0 + \mathbf{B}_k \mathbf{h} + (\mathbf{B}_k\mathbf{M}+\mathbf{E}_k)\bm\delta
\end{equation}
for suitable matrices $\mathbf{B}_k\in\R^{n_x \times Tn_u}$ and $\mathbf{E}_k\in\R^{n_x \times Tn_\delta}$.
The chance constrained MPC problem is now given as
\begin{align}\label{eq:SMPC_formulation}
	\min_{\mathbf{M},\mathbf{h}} & \quad  J(\mathbf{M},\mathbf{h}) \\
	\text{s.t.} & \quad u_k(\bm\delta) \in\mathbb{U} \qquad \forall \bm\delta\in\Delta^{T},\ \forall k\in\{0,\ldots,T-1\}, \nonumber \\
			& \quad \bbP \left[ x_{k}(\bm\delta) \in\mathbb{X} \right] \geq 1-\epsilon \quad\ \forall k\in\{1,\ldots,T\},  \nonumber
\end{align}
with $u_k(\bm\delta)$ as in \eqref{eq:input_k} and $x_k(\bm\delta)$ as in \eqref{eq:state_k}. Note that we require the input constraints in \eqref{eq:SMPC_formulation} to be satisfied robustly for every uncertainty realization since they are typically inflexible. It is shown in \cite{ben2004adjustable, Goulart2006} that if $\Delta$ is a  polytope, then the robust input constraints admit an exact reformulation as a set of linear constraints. 

The RMPC counterpart of \eqref{eq:SMPC_formulation} is obtained by replacing the chance constraints with sampled constraints as described next: for each stage $k$, we extract the multi-sample $\bm{\omega}_k := \{\bm{\delta}^{(1)},\ldots,\bm{\delta}^{(N_k)}\}$.  Then the $k$th chance constraint is replaced with $N_k$ sampled constraints, each corresponding to a sample $\bm{\delta}^{(i)}\in\bm\omega_k$. Clearly, the sample sizes $N_k$ depend on the Helly's dimensions of each stage, denoted by $\zeta_k$. 

\subsection{Comparison of Bounds on Helly's Dimension}\label{sec:comparisonBoundsGeneral}
As reported in \cite[Section III]{zhang:margellos:goulart:lygeros:13}, the \textit{standard bound} for $\zeta_k$ is given by $\zeta_k^\text{std} := kn_u + n_u n_\delta \frac{k(k-1)}{2}$, whereas the \textit{s-rank} gives the bound $\zeta_k^\text{s-r} := \min\left\{\text{rank}(F), kn_u\right\} + n_u n_\delta \frac{k(k-1)}{2}$ \cite[Proposition 1]{zhang:grammatico:schildbach:goulart:lygeros:14ecc}. Clearly, both bounds scale quadratically in the number of stages as $\mathcal{O}(k^2 n_u n_\delta)$. This renders the application of RMPC numerically challenging due to the large number of sampled constraints that need to be stored and processed when solving the problem. For example, \cite{zhang:margellos:goulart:lygeros:13} reports that when using these bounds, only small problems ($n_u=n_\delta=1$, $T\leq20$, $\epsilon=0.1$) can be solved efficiently.

In contrast, our new bound based on Proposition \ref{prop:convex_affine_r_sc} suggests that $\zeta_k\leq\zeta_k^\text{new} := n_f(kn_\delta+1)$ for the general case, and $\zeta_k\leq\zeta_k^\text{new} := \frac{n_f}{2}(kn_\delta+1)$ if the state constraints are just upper and lower bounded (Proposition \ref{prop:sc_mults_box}). Note that our new bound $\zeta_k^\text{new}$ scales \emph{linearly} in the number of stages $k$, and can therefore be applied to problems with a long horizon. Moreover, our bound does not depend on the number of inputs, but rather on the number of state constraints.  
The second column in Table~\ref{tab:sampleSizes} summarizes the different bounds.

\begin{table}[htb]
\tabcolsep = 1.0mm
\caption{Comparison of bounds on $\zeta_k$ for general chance constrained MPC of form \eqref{eq:SMPC_formulation} and the inventory control example in Section \ref{sec:InventoryExample}, where $\alpha := \min\left\{  \text{rank}(F), kn_u  \right\}.$}
\label{tab:sampleSizes}
\begin{center}
\begin{tabular}{c||c |c }
\toprule
Bound on $\zeta_k$ & General MPC  &  Inventory Control  \\
 \midrule
standard  \cite{campi:garatti:08} &  $kn_u + n_u n_\delta \frac{k(k-1)}{2}$ & $\mathcal{O}(k^2 n_u)$  \\
 s-rank \cite{schildbach:fagiano:morari:13} & $\alpha + n_u n_\delta \frac{k(k-1)}{2}$ & $\mathcal{O}(k^2 n_u)$   \\
Proposition \ref{prop:convex_affine_r_sc} & $n_f(kn_\delta+1)$  & $\mathcal{O}(k)$ \\
\bottomrule
\end{tabular}
\end{center} 
\end{table}

\subsection{Inventory Control Example}\label{sec:InventoryExample}
\ifTwoColumn
	As an example for \eqref{eq:SMPC_formulation}, let us now consider an inventory management problem for a warehouse supplied by $n_u$ factories. The inventory dynamics are $x_{k+1} = x_k + \mathbf{1}^\top u_k - v_k - \delta_k$, where $x_k\in\R$ is the inventory level, $u_k\in\R^{n_u}$ the factory outputs, $v_k\in\R$ the nominal demand, $\delta_k\in\R$ the demand uncertainty, and $\mathbf{1}\in\R^{n_u}$ the all-one vector. 
\fi
Following \cite{ben2004adjustable}, we let $v_k = 300\left( 1 + 0.5\sin\left( \pi k/12 \right) \right)$. The demand uncertainties $\delta_k$ are assumed to be i.i.d.\ random variables uniformly distributed on $\Delta=[-200,200]$. We consider state and input constraints of the form $\bbP[500 \leq x_k ] \geq 1-\epsilon$ and $0\leq u_k \leq 567$, with $u_k$ as in \eqref{eq:input_k}, and minimize the cost $J(\mathbf{u}) = \sum_{k=0}^{T-1} \mathbb{E}\left[ 100\, x_k + k \mathbf{1}^\top u_k \right] + \mathbb{E}[100\, x_T]$.	
		
\subsubsection{Bounding Helly's Dimension}
Figure \ref{fig:sdimension} depicts the bounds on the number of support constraints for the s-rank and the new bound\footnote{The standard bound is not plotted since it is never better than the s-rank, see also Table \ref{tab:sampleSizes}.}  as a function of $n_u$ and $k$. It illustrates well the linear and quadratic growth of the s-rank with respect to the number of inputs and stages, respectively. In contrast, the new bound scales moderately along the stages, and is entirely independent of the number of inputs. Figure \ref{fig:sdimension} also demonstrates that while the s-rank provides good bounds for small values of $k$ and $n_u$, it is outperformed by our new bound when these values become large. This asymptotic behavior is summarized in the third column of Table  \ref{tab:sampleSizes}. This observation is indeed a typical behavior of these two bounds, and hold generally for RMPC problems of this kind.

\ifTwoColumn
	
	\begin{figure}[thb]
	\centering
	\includegraphics[trim = 0mm 0mm 0mm 0mm, clip, width=0.5\textwidth]{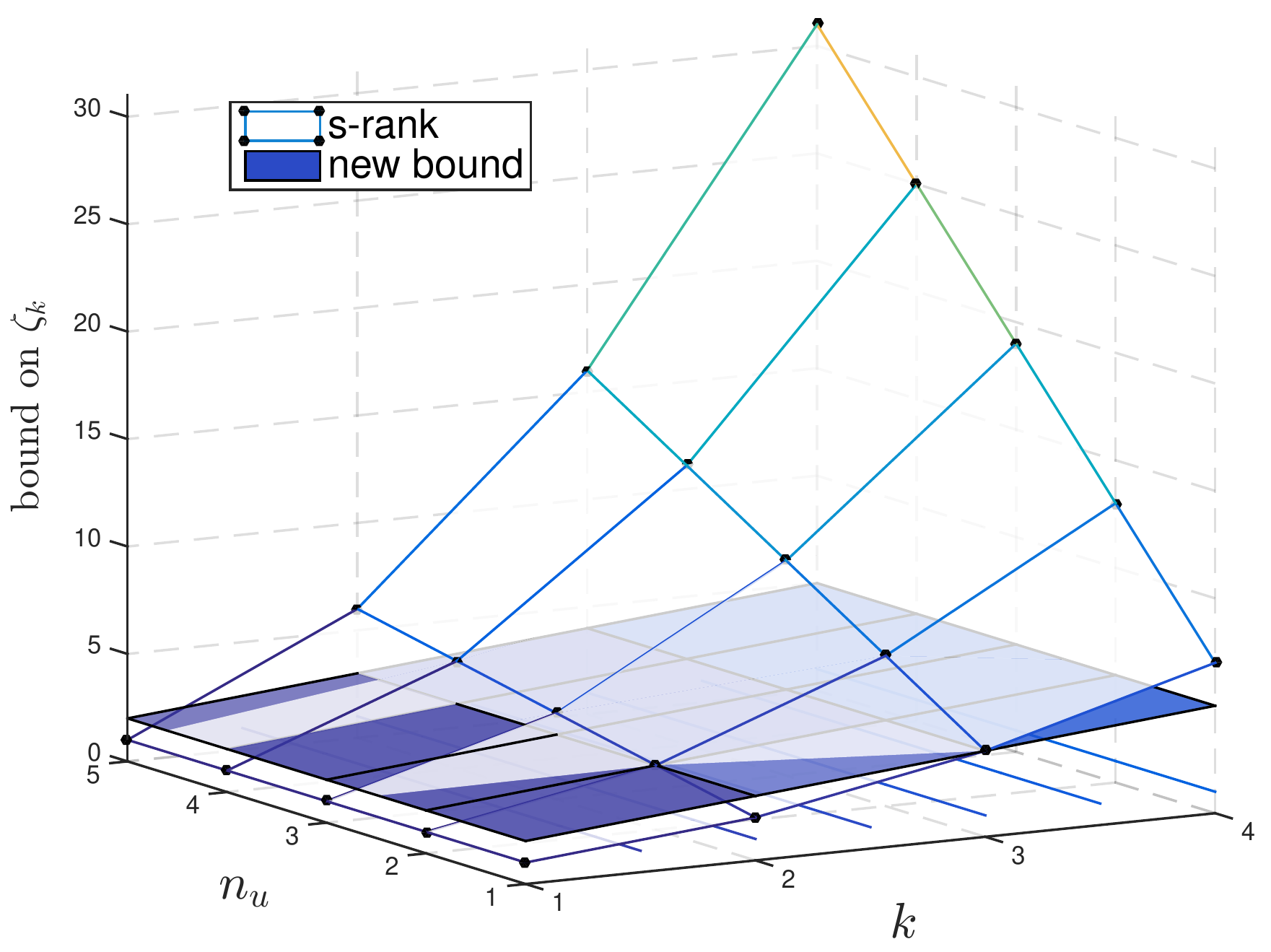} \\
	\caption{Bound on Helly's dimension using the s-rank and new bound (Proposition \ref{prop:convex_affine_r_sc}) for $k\in\{1,\ldots,4\}$ and $n_u\in\{1,\ldots,5\}$.}
	\label{fig:sdimension}
	\end{figure}
\fi

\subsubsection{Empirical Confidence and Violation Probability}
For $n_u=5$ and $T=15$, Figure \ref{fig:beta} depicts empirical estimates of the confidence $\hat\beta_k$ along the stages for both the s-rank and the new bound, while Figure \ref{fig:epsilon} shows the empirical violation probability $\hat\epsilon_k$. Note that, for reasons of computational tractability, different values $\epsilon$ and $\beta$ were used in the two figures. The sample sizes $N_k$ used to generate both plots were acquired by numerically inverting \eqref{eq:N_choice_zeta} with $\zeta_k^\text{new}$ and $\zeta_k^\text{s-r}$ in place of $\zeta$, respectively.

Figures \ref{fig:beta} and \ref{fig:epsilon} show that the new bound is closer to the predefined confidence and violation probability levels (red) than the s-rank whenever $k\geq2$, but that the s-rank outperforms the new bound only for $k=1$.  This is not surprising, since for $k=1$, the s-rank provides a tighter bound on $\zeta_1$ than the new bound does, see also Figure \ref{fig:sdimension} and Table  \ref{tab:sampleSizes}. From these figures we conclude that for most stages, the new bound is less conservative than the s-rank, and therefore also the standard bound.

	\begin{figure}[!t]
	\centering
	\includegraphics[trim = 7mm 5mm 5mm 5mm, clip, width=0.5\textwidth]{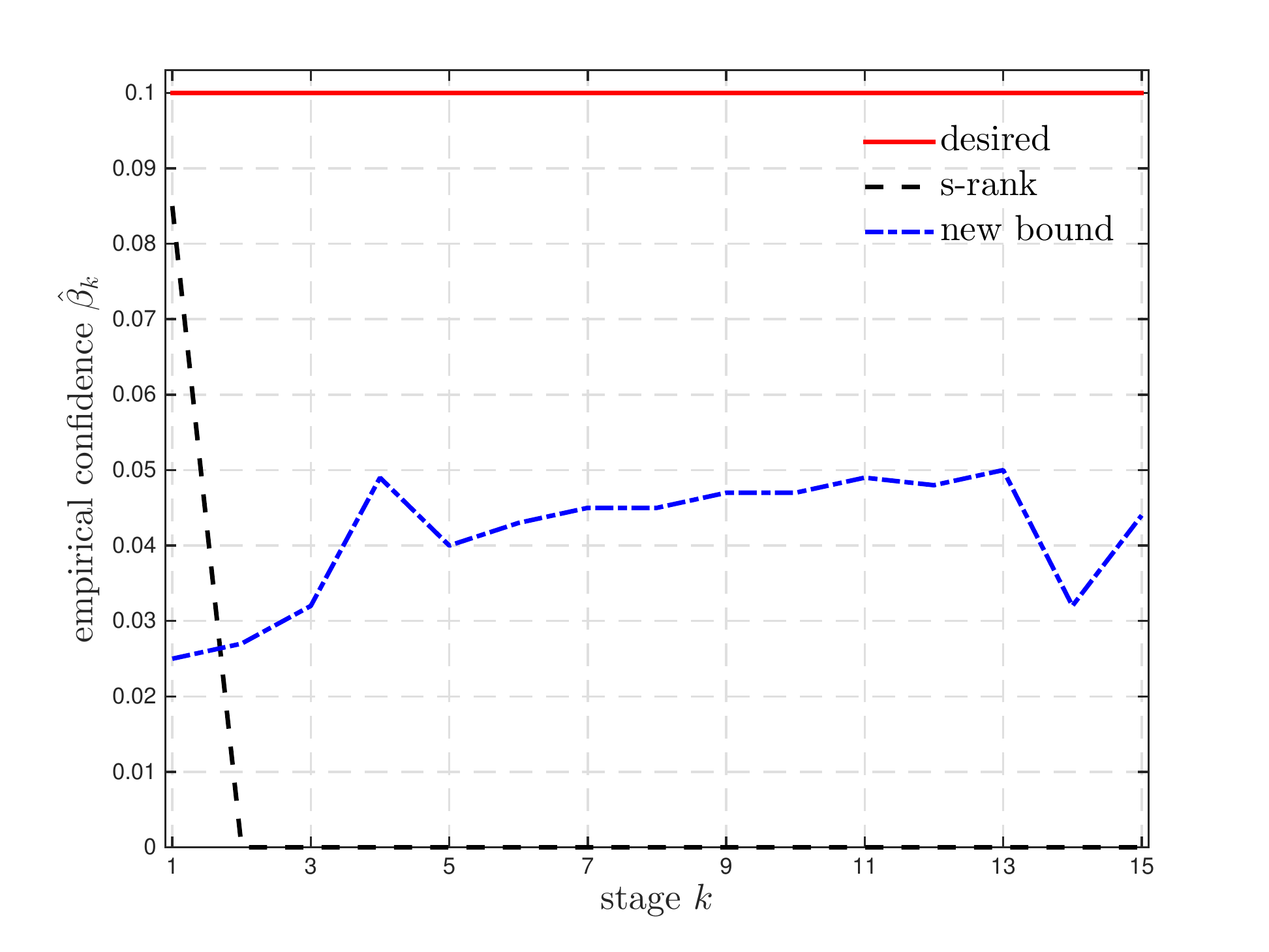} \\
	\caption{Predefined $\beta$ (red), its empirical estimate using the s-rank (black) and the new bound (blue) for $n_u=5$, $\epsilon=0.2$ and $\beta=0.1$. The estimates are obtained from $10^4$  instances, where for each instance the empirical violation probabilities are estimated by evaluating the solutions of each RMPC problem against $10^4$ uncertainty realizations. Then, $\hat\beta$ is determined as the fraction $\hat\epsilon$ which were higher than the predetermined level of $\epsilon=0.2$.}
	\label{fig:beta}
	\end{figure}

	\begin{figure}[!t]
	\centering
	\includegraphics[trim = 8mm 5mm 5mm 5mm, clip, width=0.5\textwidth]{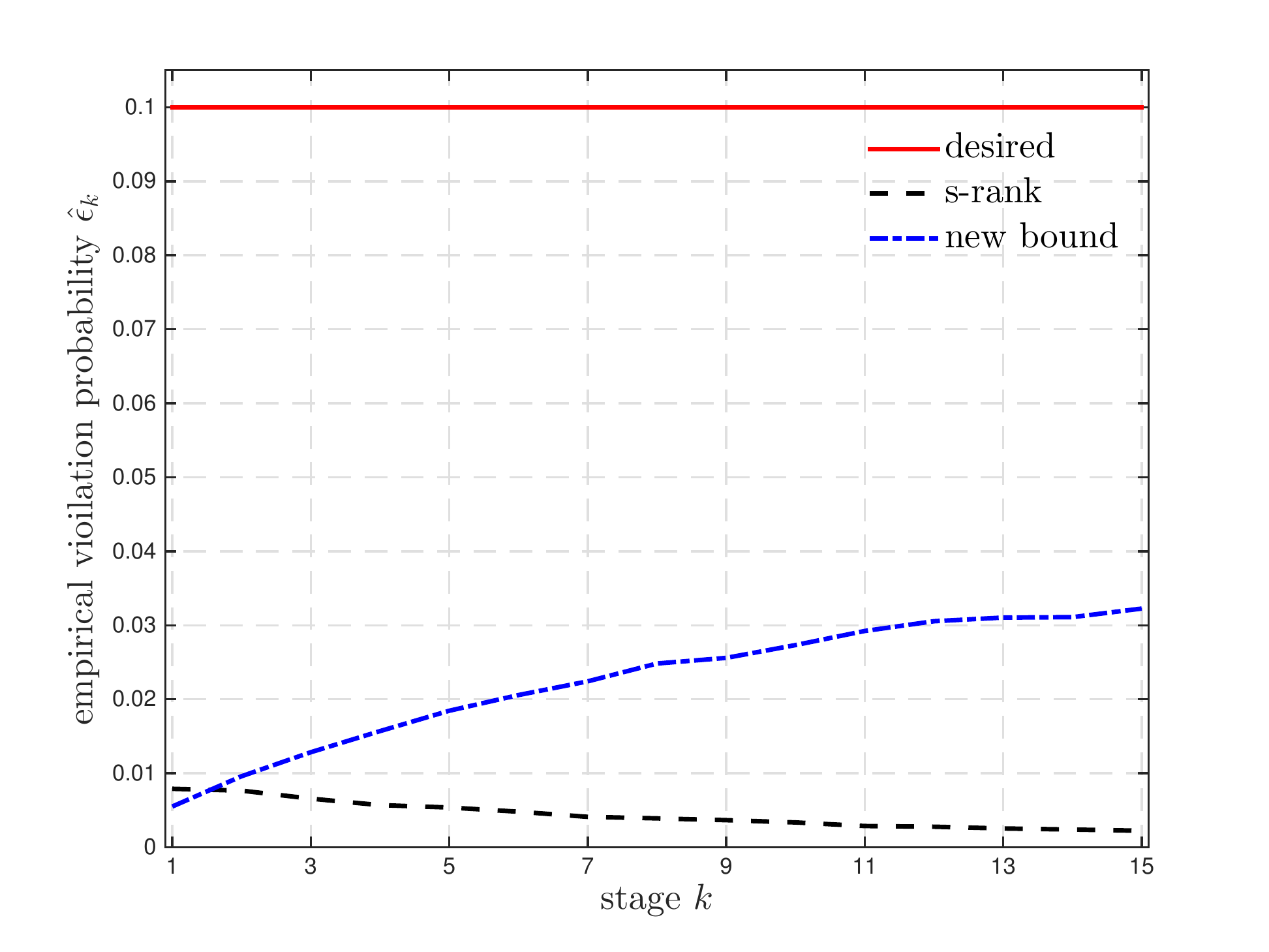} \\
	\caption{Predefined $\epsilon$ (red), its empirical estimate using the s-rank (black) and the new bound (blue) for $n_u=5$, $\epsilon=0.1$ and $\beta=10^{-7}$. The estimates are obtained by averaging over the empirical violation probabilities of $10^3$  instances, each estimated by evaluating its solution against $10^4$ uncertainty realizations. }
	\label{fig:epsilon}
	\end{figure}

\subsubsection{Tightness of the Bounds}
Even if the new bound improves upon the existing ones, Figures \ref{fig:beta} and \ref{fig:epsilon} suggest they are still conservative. Indeed, the empirical violation probabilities $\hat\epsilon_k$ can be improved using a so-called sampling-and-discarding method, in which a larger number of samples is drawn, but some of them are discarded a posteriori \cite{campi:garatti:11, calafiore:10}. In general, this method not only improves the objective value, but also brings the empirical violation probability closer to the predefined level. 
Furthermore, we can improve the empirical confidence $\hat\beta_k$ using the bound $\tilde\zeta_k^\text{new}=k$ instead of $\zeta_k^\text{new}=k+1$ from \cite[Proposition 2]{zhang:grammatico:schildbach:goulart:lygeros:14ecc}, since both assumptions required there (absolutely continuous distribution and $G(x^\star)\neq0$) are satisfied here. 

\section{Conclusion}\label{sec:conclusion}
We presented a new upper bound on Helly's dimension for random convex programs in which the constraint functions can be separated between the decision and uncertainty variables. As a consequence, the number of scenarios can be reduced significantly for problems where the dimension of the uncertain variable is smaller than the dimension of the decision variable. This leads to less conservative solutions, both in terms of cost and violation probability, as well as a reduction in the computational complexity of solutions based on the scenario approach. We applied our bounds to Stochastic MPC problems for chance constrained systems, and demonstrated  the quality of the bound on an inventory example. We believe that both theoretical and applied research in the field of Randomized MPC for uncertain linear systems can benefit from our results.

\section*{Acknowledgments}
The authors thank Dr.\ Kostas Margellos and Dr. Angelos Georghiou for fruitful discussions on related topics.  
Research was partially funded by the Swiss Nano-Tera.ch under the project HeatReserves and the European Commission under project SPEEDD (FP7-ICT 619435).

\begin{appendix}

\section{Proofs}\label{app: proofs}
For ease of presentation we first show the proof of Proposition \ref{prop:sc_mults} and then the proof of Lemma \ref{lem:convex_sep_r_sc}.

\subsubsection*{Proof of Proposition \ref{prop:sc_mults}}\label{sec:cor_mults}
Let us first prove the following auxiliary statement.
\begin{claim*}
	If $r=1$, then $|\textnormal{sc}(\textnormal{SP}[\omega])| \leq m$.
\end{claim*}
\begin{proof}
Consider $\text{SP}[\omega]$ subject to $N\geq m+1$ sampled constraints, generated by $\omega = \{\delta^{(1)},\ldots,\delta^{(N)}\}$. Suppose, for the sake of contradiction, that there exists $m+1$ support constraints. Without loss of generality, they are generated by the first $m+1$ samples, i.e.\ $\text{sc}(\text{SP}[\omega])=\{\delta^{(1)},\ldots,\delta^{(m+1)}\}$. Let $x^\star$ be the optimal solution of $\text{SP}[\omega]$ and 
$x_i^\star$ be the optimal solution of $\text{SP}[\omega\setminus\{\delta^{(i)}\}]$, so that $c^\top x^\star_i < c^\top x^\star$. If we define $p_0:=G(x^\star),\ p_1 := G(x^\star_1),\ \ldots,\ p_{m+1} := G(x^\star_{m+1})$, then $\mathcal{P}:=\big\{p_0,p_1,p_2,\ldots,p_{m+1}\big\}$ is a collection of $m+2$ vectors in $\R^{m}$. By Radon's Theorem, there exist two index sets $K,L\subseteq\{0,\ldots,m+1\},\ K\cup L=\{0,\ldots,m+1\},\ K\cap L=\varnothing$, such that $\mathcal{P}$ can be partitioned into two disjoint sets $\mathcal{P}_K := \{p_k \mid k\in K\}$ and $\mathcal{P}_L := \{p_l \mid l\in L\}$, with $\text{conv}(\mathcal{P}_K) \cap \text{conv}(\mathcal{P}_L)\neq\varnothing$. Therefore there exists $\bar{p}\in\text{conv}(\mathcal{P}_K) \cap \text{conv}(\mathcal{P}_L) \subseteq \R^m$. 
Because $\bar{p}\in\text{conv}(\mathcal{P}_K)$, there exist non-negative scalars $\{\bar\kappa_k\}_{k\in K}$, $\sum_{k\in K}\bar\kappa_k = 1$, such that $\bar{p} = \sum_{k\in K} \bar\kappa_k p_k$. Moreover, for all $k\in K$ it holds that $p_k^\top q(\delta^{(l)}) + s(\delta^{(l)})\leq0$ $\forall l\in L\setminus\{0\}$. Therefore, since $\bar{p}\in\text{conv}(\mathcal{P}_K)$ and $\sum_{k\in K}\bar\kappa_k = 1$, it holds $\bar{p}^\top q(\delta^{(l)}) + s(\delta^{(l)})\leq0$ $\forall l\in L\setminus\{0\}$. Since $\bar{p}$ also belongs to $\text{conv}(\mathcal{P}_L)$, we can analogously conclude that $\bar{p}^\top q(\delta^{(k)}) + s(\delta^{(k)})\leq0$ $\forall k\in K\setminus\{0\}$. Hence, since $K\cup L \supseteq \{1,\ldots,m+1\}$, it holds that $\bar{p}^\top q(\delta^{(i)}) + s(\delta^{(i)})\leq0$ $\forall i\in \{1,\ldots,m+1\}$. 
\ifTwoColumn
   	{}
\fi
Without loss of generality, we now assume that $0\not\in K$, i.e.\ $G(x^\star) \not\in \mathcal{P}_K$ and let $\bar{x} := \sum_{k\in K} \bar\kappa_k x_k^\star$. 

From the convexity of $g(\cdot,\delta)$, for every $i\in\{1,\ldots,N\}$ it holds that $g(\bar{x},\delta^{(i)})\leq\sum_{k\in K}\bar\kappa_k g(x^\star_k,\delta^{(i)})=\sum_{k\in K}\bar\kappa_k \left[ p_k^\top q(\delta^{(i)}) + s(\delta^{(i)}) \right]=\bar{p}^\top q(\delta^{(i)}) + s(\delta^{(i)}) \leq 0$, i.e., $\bar{x}$ is a feasible point of $\text{SP}[\omega]$. Since $c^\top x_i^\star < c^\top x^\star$ $\forall i\in\{1,\ldots,m+1\}$, we have $c^\top\bar{x} < c^\top x^\star$. This, however, contradicts the fact that $x^\star$ is the optimizer of $\text{SP}[\omega]$. 
\end{proof}
Let now $r\geq1$. From the above claim it holds that the number of support constraint for each row is at most $m$. Therefore, with $rm$ constraints, there are at most $rm$ support constraints. This concludes the proof. \qed

\subsection*{Proof of Lemma \ref{lem:convex_sep_r_sc}}
The constraint $G(x)\,q(\delta) + H(x) + s(\delta) \leq 0$ can be expressed as $g(x,\delta)=\tilde{G}(x)\, \tilde{q}(\delta)+s(\delta)\leq0$ by choosing $\tilde{G}(x):=[G(x) \ H(x) ]$ and $\tilde{q}(\delta):=[q(\delta) ; \mathbf{1}]\in\R^{m+1}$. Invoking Proposition \ref{prop:sc_mults}, we immediately get at most $r(m+1)$ support constraints. \qed

\ifTwoColumn
	\subsection*{Proof of Proposition \ref{prop:convex_quad_r_sc}}
	We proceed in two steps and first prove the  statement.
	\begin{claim*}
	If $r=1$, then $|\textnormal{sc}(\textsc{SP}[\omega])| \leq d(d+3)/2+1$.
	\end{claim*}
	\begin{proof}
		The term $\delta^\top A(x) \delta$ can be equivalently written as ${p}_1(x)^\top {q}_1(\delta)$ for some properly defined ${p}_1(x)$ and ${q}_1(\delta)$. Indeed, if $\delta_i$ is the $i$th component of $\delta$, then ${q}_1(\delta)$ contains the auxiliary uncertainties $\delta_i\delta_j$ for $i,j=1,\ldots,d$. Because $\delta_i\delta_j=\delta_j\delta_i$, it follows from symmetry that ${q}_1(\delta)\in\R^{d(d+1)/2}$. As in the affine case, the remaining term $b(x)^\top \delta + c(x)$ can be decomposed as ${p}_2(x)^\top {q}_2(\delta)$ with $p_2(x)^\top := [b^\top(x) \ c(x)]$ and $q_2(\delta)=[\delta;1]\in\mathbb{R}^{d+1}$. Hence, the constraint can be written as $g(x,\delta)=p(x)^\top q(\delta)$, where $p(x) := [p_1(x);p_2(x)]$ and $q(\delta) := [q_1(\delta); q_2(\delta)]\in\R^{d(d+3)/2+1}$. The claim then follows from Proposition \ref{prop:sc_mults} with $m=d(d+3)/2+1$ and $s(\delta)=0$.
	\end{proof}
	We now let $r\geq1$. Since each constraint $g_i$ has at most $d(d+3)/2+1$ support constraints, with $r$ constraints there are at most $rd(d+3)/2+r$ support constraints.\qed
\fi

\subsection*{Proof of Proposition \ref{prop:sc_mults_box}}
We proceed similar as in the proof of Proposition \ref{prop:sc_mults}.
\begin{claim*}
	If $r=1$, then $|\textnormal{sc}(\textnormal{SP}[\omega])| \leq m$.
\end{claim*}
\begin{proof}[Proof (Sketch).] The first part is identical to the first paragraph of the proof of the claim in Proposition \ref{prop:sc_mults}  with the difference that $\bar{p}=\sum_{k\in K} \bar{\kappa}_k p_k $ satisfies $\underline{g} \leq \bar{p}^\top q(\delta^{(i)}) + s(\delta^{(i)}) \leq \bar{g}\ \forall i\in\{1,\ldots,m+1\}$. 
Moreover, by affinity of $g_\textnormal{mult}(\cdot,\delta)$  it follows $g_\textnormal{mult}(\bar{x},\delta^{(i)})=\sum_{k\in K}\bar{\kappa}_k g_\textnormal{mult}(x^\star_k,\delta^{(i)}) = \bar{p}^\top q(\delta^{(i)}) + s(\delta^{(i)})$, so that for every $i\in\{1,\ldots,N\}$ we have $\underline{g} \leq  g_\textnormal{mult}(\bar x,\delta^{(i)})  \leq \overline{g}$, i.e.\ $\bar{x}$ is a feasible point of SP$[\omega]$ with cost $c^\top \bar{x} < c^\top x^\star$. This, however, contradicts the fact that $x^\star$ is the optimizer of SP$[\omega]$.
\end{proof}
The proposition follows immediately since $r$ such constraints can generate at most $rm$ support constraints. \qed


\end{appendix}
\balance
\bibliographystyle{unsrt}
\bibliography{library1}

\end{document}